\documentclass[10pt]{article}

\usepackage[latin1]{inputenc}
\usepackage{amsfonts}
\usepackage{amsmath}
\usepackage{amssymb}
\usepackage{amsthm}
\usepackage{fontenc}
\usepackage{graphicx}
\usepackage[margin=1.2in]{geometry}
\usepackage[all]{xy}
\usepackage{hyperref}
\newtheorem{theorem}{Theorem}[section]
\newtheorem{lemma}[theorem]{Lemma}
\newtheorem{proposition}[theorem]{Proposition}
\newtheorem*{theorem*}{Theorem}

\newtheorem{corollary}[theorem]{Corollary}

\newtheorem{question}{Question}

\numberwithin{equation}{section}

\begin{document}

\title{Hyperplane Equipartitions Plus Constraints}
\author{Steven Simon}%


\date{}

\maketitle

\begin{abstract}

While equivariant methods have seen many fruitful applications in geometric combinatorics, their inability to answer the now settled Topological Tverberg Conjecture has made apparent the need to move beyond the use of Borsuk--Ulam type theorems alone. This impression holds as well for one of the most famous problems in the field, dating back to 1960, which seeks the minimum dimension $d:=\Delta(m;k)$ such that any $m$ mass distributions in $\mathbb{R}^d$ can be simultaneously equipartitioned by $k$ hyperplanes. Precise values of $\Delta(m;k)$ have been obtained in few cases, and the best-known general upper bound $U(m;k)$ typically far exceeds the conjectured-tight lower bound arising from degrees of freedom. Following the ``constraint method'' of Blagojevi\'c, Frick, and Ziegler originally used for Tverberg-type results and recently to the present problem, we show how the imposition of further conditions -- on the hyperplane arrangements themselves (e.g., orthogonality, prescribed flat containment) and/or the equipartition of additional masses by successively fewer hyperplanes (``cascades") -- yields a variety of optimal results for constrained equipartitions of $m$ mass distributions in dimension $U(m;k)$, including in dimensions \textit{below} $\Delta(m+1;k)$, which are still extractable via equivariance. Among these are families of exact values for full orthogonality as well as cascades which maximize the ``fullness" of the equipartition at each stage, including some strengthened equipartitions in dimension $\Delta(m;k)$ itself. 

\end{abstract}

\section{Introduction}

\subsection{Historical Summary}

	With the recent negative resolution [3] to the Topological Tverberg Conjecture [2], perhaps the most famous remaining open question in topological combinatorics is the hyperplane mass equipartition problem, originating with Gr\"unbaum [12] in 1960 and generalized by Ramos [18] in 1996:
	
	\begin{question}\label{q:1} [Gr\"unbaum--Ramos]  What is the minimum dimension $d:=\Delta(m,k)$ such that any $m$ mass distributions $\mu_1,\ldots, \mu_m$ on $\mathbb{R}^d$ can be simultaneously equipartitioned by $k$ hyperplanes? \end{question}  
	
		By a mass distribution $\mu$ on $\mathbb{R}^d$, one means a positive, finite Borel measure such that any hyperplanes has measure zero (e.g., if $\mu$ is absolutely continuous with respect to Lebesgue measure). To say that $k$ hyperplanes $H_1,\ldots, H_k$ equipartition $\mu$ means that
\begin{equation}\label{eq:1.1} \mu(\mathcal{R}_g)= \frac{1}{2^k}\mu(\mathbb{R}^d)\end{equation} 
\noindent for all $g=(g_1,\ldots, g_k) \in \mathbb{Z}_2^{\oplus k}$, where the $\mathcal{R}_g:=H_1^{g_1}\cap \cdots \cap H_k^{g_k}$  are the regions obtained by intersecting the resulting half-spaces $H_i^0:=\{\mathbf{u}\in \mathbb{R}^d\mid \langle \mathbf{u},\mathbf{a}_i\rangle\geq b_i\}$ and $H_i^1:=\{\mathbf{u}\in \mathbb{R}^d\mid \langle \mathbf{u},\mathbf{a}_i\rangle\leq b_i\}$ determined by each hyperplane $H_i=\{\mathbf{u}\in \mathbb{R}^d\mid \langle \mathbf{u},\mathbf{a}_i\rangle=b_i\}$, $(\mathbf{a}_i,b_i)\in S^{d-1}\times\mathbb{R}$. Note that equipartitoning hyperplanes are distinct (and affine independent), lest $\mathcal{R}_g=\emptyset$ for some $g\in \mathbb{Z}_2^{\oplus k}$.

	The lower bound \begin{equation}\label{eq:1.2} k\Delta(m;k) \geq m(2^k-1) \end{equation} was proved by Ramos [18] via a generalization of a moment curve argument of Avis [1] for $m=1$, and the conjecture $\Delta(m;k)=L(m;k):=\left\lceil \frac{m(2^k-1)}{k}\right\rceil$ posited there has been confirmed for all known values of $\Delta(m;k)$. Owing to the reflective and permutative symmetries on $k$ hyperplanes, there is a natural action of the wreath product $\mathfrak{S}_k^\pm:=\mathbb{Z}_2\wr \mathfrak{S}_k$ on each collection of regions, $\mathfrak{S}_k$ being the symmetric group, so that upper bounds on $\Delta(m;k)$ have been obtained via equivariant topology. Using the ubiquitous ``Configuration-Space/Test-Map (CS/TM) paradigm" formalized by \v Zivaljevi\'c (see, e.g., [24]), any collection of $k$ equipartitoning hyperplanes can be identified with a zero of an associated continuous $\mathfrak{S}_k^\pm$-map $f: X\rightarrow V$, where $X$ is either the $k$-fold product or join of spheres and $V$ is a certain $\mathfrak{S}_k^\pm$-module (see Section 3 for a review of this construction). In favorable circumstances the vanishing of such maps is guaranteed by Borsuk--Ulam type theorems which rely on the calculation of advanced algebraic invariants such as the ideal-valued index theory of Fadell--Husseini [11] or relative equivariant obstruction theory. Such methods have produced relatively few exact values of $\Delta(m;k)$, however, which at present are known for \\
\\
$\bullet$ all $m$ if $k=1$ (the well-known Ham Sandwich Theorem $\Delta(m;1)=m$),\\
$\bullet$ three infinite families if $k= 2$: $\Delta(2^{q+1}+ r;2)= 3\cdot 2^q + \lfloor{3r/2}\rfloor$, $r=-1,0,1$ and $q\geq 0$ [16, 5, 6],\\
$\bullet$ three cases if $k=3$: $\Delta(1;3)=3$ [13], $\Delta(2;3)=5$ [5], and $\Delta(4;3)=10$ [5], and\\
$\bullet$ \textit{no values of m} if $k\geq 4$.\\
\\
The currently best known general upper bound on $\Delta(m;k)$, \begin{equation} \label{eq:1.3} \Delta(m;k)\leq U(m;k):=2^{q+k-1}+r \,\,\, \text{for}\,\,\, m=2^q+r,\,\,\, 0\leq r<2^q, \end{equation}
relies only on $\mathbb{Z}_2^{\oplus k}$-equivariance rather than the full symmetries of $\mathfrak{S}^\pm_k$ and was given by Mani--Levitska, Vre\'cica, and \v Zivaljevi\'c in 2007 [16].  It is easily verified that $U(m;k)=L(m;k)$ follows from \eqref{eq:1.3} only when (a) $k=1$ or when (b)  $k=2$ and $m=2^{q+1}-1$, with a widening gap between $U(m;k)$ and $L(m;k)$ as $r$ tends to zero, and as either $q$ or $r$ increases. For instance, when $m=1$ and $k=4$ one already has $4\leq \Delta(1;4)\leq 8$, which can be compared to the best-known estimate $4\leq \Delta(1;4)\leq 5$ of [5]. 

\subsection{Constrained Equipartitions} Given the present state of the problem -- and particularly in light of the failure of analogous equivariant methods to settle the Topological Tverberg conjecture in general -- it is natural to  suppose that methods beyond Borsuk--Ulam type theorems alone are necessary to settle Question \ref{q:1}. In lieu of producing new values or improved upper bounds on $\Delta(m;k)$, however, this paper presents optimal hyperplane equipartitions -- under the imposition of further constraints -- which can still be extracted from the underlying equivariant topological schema.  Our approach follows the ``constraint method" of Blagojevi\'c, Frick, and Ziegler, used in the context of Question \ref{q:1} in [6] and originally in [4] to derive a variety of optimal Tverberg-type results (colored versions, dimensionally controlled results of Van-Kampen Flores-type, and others) as direct consequences of positive answers to the topological Tverberg Conjecture itself. Whenever any $m$ masses can be equipartitioned by $k$ hyperplanes in $\mathbb{R}^d$, we shall ask for further conditions which can be imposed in the same dimension, whether (1) on the arrangement of equipartitoning hyperplanes and/or (2) for further equipartitions (by fewer hyperplanes)  given additional masses. We shall be especially concerned with those constrained equipartitions which are (a) optimal with respect to the original Gr\"unbaum--Ramos problem, in the sense that $d<\Delta(m+1;k)$ (including when $d=\Delta(m;k)$) and/or (b) those which are tight in that the total number of equipartition and arrangement conditions matches the full $kd$ degrees of freedom for $k$ hyperplanes in $\mathbb{R}^d$ (see, e.g. [10] for a result in a similar spirit of the later). While a variety of interesting constraints along the lines of (1) and (2) can be imposed, we shall focus on (combinations of) the following three:

\begin{itemize} 
\item (i) (Orthogonality) Given a subset $\mathcal{O}\subseteq \{(r,s)\mid 1\leq r <s\leq k\}$, can one ensure that $H_r\perp H_s$ for each $(r,s)\in\mathcal{O}$? In particular, can one  guarantee that all the hyperplanes are pairwise orthogonal?
\item (ii) (Prescribed Flat Containment)  Given affine subspaces $A_1,\ldots, A_k$, can one ensure that $A_i\subseteq H_i$ for each $1\leq i \leq k$? In particular, can some or all of the $A_i$ be prescribed \textit{linear}, or, given a \textit{filtration} $A_1\subseteq \cdots \subseteq A_k$, can one guarantee that $A_i\subseteq(H_i\cap \cdots \cap H_k)$ for each $i$?
\item (iii) (``Cascades'') In addition to the full equipartition of any collection of $m=m_1$ masses by $H_1,\ldots, H_k$, can the hyperplanes be chosen so that $H_2,\ldots, H_k$ also equipartition any additional collection of $m_2$  prescribed masses, that $H_3,\ldots, H_k$  equipartition any further given collection of $m_3$ masses, and so on, until $H_k$ equipartitions (bisects) any additionally given $m_k$ masses? In particular, can one maximize the ``fullness" of the equipartition at each stage of the cascade so that, for any $1\leq i \leq k$ increasing $m_i$ to $m_i +1$ while simultaneously setting $m_j=0$ for all $j>i$ requires a corresponding increase in $d$?
\end{itemize}

\noindent \textbf{Example 1.} Although $\Delta(2^{q+1}-1;2)=3\cdot 2^q -1$, this result is not tight because there is still one remaining degree of freedom in \eqref{eq:1.2}. Thus one can ask that (i) the equipartitioning hyperplanes are orthogonal, or that (iii) one of the hyperplanes bisects any further prescribed mass (and so, by considering a ball with uniform density, can be made to (ii) contain any prescribed point). That either of these conditions can always be imposed is stated in the second line of Theorem \ref{thm:1.1}  below.\\
\\
\noindent \textbf{Example 2.} For a low dimensional  cascade, consider an equipartition of a single mass $\mu_1$ by three hyperplanes $H_1,H_2$ and $H_3$. Theorem \ref{thm:2.1} below shows that requiring $H_2$ and $H_3$ to equipartition any second prescribed mass $\mu_2$ requires $d\geq 4$, and since $\Delta(1;3)=3$ while $\Delta(2;3)=5$, such an equipartition in dimension $U(1;3)=4$ would be optimal with respect to the original Gr\"unbaum--Ramos Problem. To ensure tightness in \eqref{eq:2.1}, one could stipulate that $H_3$ bisects any given two further masses $\mu_3$ and $\mu_4$. Such a cascade would be maximal at each stage, and, as $H_3$ would bisect each of the four masses, would represent a strengthened Ham Sandwich Theorem in $\mathbb{R}^4$ (Corollary \ref{cor:6.3} below gives a more general extension of that theorem whenever $d$ is a power of two). As a mix of (i) and (iii) above, one could require instead of the bisection of $\mu_3$ and $\mu_4$ by $H_3$ that the later be orthogonal to each of $H_1$ and $H_2$.  The existence of these constrained equipartitions, as well as the optimal result that any mass in $\mathbb{R}^4$ can be equipartitioned by three pairwise orthogonal hyperplanes, is given in the fourth line of Theorem \ref{thm:1.1}. \\

	Consolidating the three constraints above, one has the following generalization of Question \ref{q:1}, which is recovered when $\mathbf{a}=\mathbf{0}$ and $\mathcal{O}=\emptyset$:

\begin{question} \label{q:2} Let $\mathbf{m}=(m_1,\ldots, m_k)\in \mathbb{N}^k$ with $m_1\geq 1$, let $\mathbf{a}=(a_1,\ldots, a_k)\in \mathbb{N}^k$, and let $\mathcal{O}\subseteq \{(r,s)\in \mathbb{Z}_+^2 \mid 1\leq r <s\leq k\}$. What is the minimum dimension $d:=\Delta(\mathbf{m},\mathbf{a},\mathcal{O}; k)$ such that, given any family of $m=\sum_{i=1}^k m_i$ measures $\{\mu_{i,j}\}_{1\leq i\leq k, 1\leq j \leq m_i}$ in $\mathbb{R}^d$ and given any $k$ affine subspaces $A_1,\ldots, A_k$, $-1\leq \dim(A_i)=a_i-1$ for all $1\leq i \leq k$, there exist $k$ hyperplanes $H_1,\ldots,H_k$ such that \begin{itemize} \item(i) $H_r\perp H_s$ for all $(r,s)\in \mathcal{O}$, \item (ii) $A_i\subseteq H_i$ for all $1\leq i \leq k$, and \item (iii) $H_i,\ldots, H_k$ equipartitions $\mu_{i,1},\ldots, \mu_{i,m_i}$ for each $1\leq i \leq k$? \end{itemize} \end{question}

In the special case that $|\mathcal{O}|=\binom{k}{2}$, we shall let $\Delta^\perp(m;k):=\Delta((m,0,\ldots,0),\mathbf{0},\mathcal{O};k)$ denote the ``full orthogonal" generalization of Question \ref{q:1} previously considered in [7] and [18]. Likewise, we shall let $\Delta(\mathbf{m};k) = \Delta(\mathbf{m},\mathbf{0},\emptyset ;k)$ denote a ``pure" cascade, while $\Delta^\perp (\mathbf{m}; k)$ will denote a cascade with full orthogonality.

\subsection{Summary of Results} 

The remainder of the paper is structured as follows. In Section 2, we give (Theorem \ref{thm:2.1}) an extension of the Ramos lower bound for $\Delta(\mathbf{m},\mathbf{a},\mathcal{O};k)$.  Following a detailed review of the CS/TM paradigm in Section 3 as previously applied to Question \ref{q:1}, we discuss in Section 4 its modifications to our constrained cases, the heart of which (as in [4]) lies in the imposition of further conditions on the target space. In Section 5, we show how a reduction trick (Lemma \ref{lem:5.1}) from [6] immediately implies that upper bounds $\Delta(m+1;k)\leq d+1$ obtained via that scheme produce upper bounds $\Delta^\perp(m;k)\leq d$ for full orthogonality (Theorem \ref{thm:5.2}), an observation we owe to Florian Frick.	This restriction technique does not produce cascades or specified affine containment, however, and answers to Question \ref{q:2} under a variety of these constraints, including partial or full orthogonality, are given in Section 6 using cohomological methods  (Theorems \ref{thm:6.2} and \ref{thm:6.4}, Propositions \ref{prop:6.5}--\ref{prop:6.6}). In all of these results one has $\Delta(\textbf{m},\textbf{a},\mathcal{O};k)=U(m_1;k)$, thereby strengthening the best-known upper bound $\Delta(m_1;k)\leq U(m_1;k)$ for the original problem, and unlike those of Theorem \ref{thm:5.2} in general, each is tight with respect to degrees of freedom. While a variety of examples are given in Section 7 (Corollaries \ref{cor:7.1}--\ref{cor:7.3}), we collect below a sampling of tight, optimal, and/or maximal constrained equipartitions using both approaches, as well as two additional estimates of interest.

\begin{theorem} \label{thm:1.1} Let $q\geq 0$ and let $\langle k \rangle^\perp = \{(r,k)\mid 1\leq r<k\}$.\\
$\bullet$ $\Delta^\perp(2^{q+1}; 2)=3\cdot 2^q+1$\\
$\bullet$ $\Delta^\perp(2^{q+1}-1;2)= \Delta((2^{q+1}-1,1);2) =3\cdot 2^q -1$\\
$\bullet$ $\Delta((2^{q+2}-2, 2);2)=\Delta^\perp((2^{q+2}-2),1); 2))=\Delta^\perp(2^{q+2}-2;2)=3\cdot 2^{q+1}-2$\\
$\bullet$ $\Delta((1,1,2);3)=\Delta((1,1,0),\langle 3\rangle^\perp;3)=\Delta^\perp(1;3)=4$\\
$\bullet$ $\Delta^\perp((2,1,2);3)=\Delta((2,2,2), \langle 3\rangle^\perp;3)=8$\\
$\bullet$ $6\leq \Delta^\perp (1;4) \leq 8$\, \text{and}\, $\Delta((1,1,2,1),\langle 4\rangle^\perp ;4)=8$\\
$\bullet$ $8\leq \Delta^\perp(3;3)\leq 9$\ \text{and}\, $\Delta((3,1,1), \langle 3 \rangle^\perp;3)=9$ 

\end{theorem}	

 We note that all of the computations above are new except for the very classical result $\Delta^\perp(1;2)=2$ (see, e.g., [9]) which follows from the intermediate value theorem.  Finally, it is worth remarking that, owing to the absence of full $\mathfrak{S}_k^\pm$-equivariance for Question \ref{q:2} in general, the results from the final two sections arise only from $\mathbb{Z}_2^{\oplus k}$-equivariance and careful polynomial computations arising from the cohomology of real projective space (Proposition \ref{prop:6.1}). Thus despite the present state of Question \ref{q:1} itself, there is still considerable remaining power even of classical topological techniques in mass partition problems more generally, a theme which we return to in [20].

\section{Geometric Lower Bounds} 

	Before proceeding to topological upper bounds for Question \ref{q:2}, we first prove the expected generalization of the Ramos lower bound $k\Delta(m;k)\geq m(2^k-1)$. As opposed to the moment curve argument given in [18], however, our proof relies only on genericity of point collections and so perhaps more simply captures the intuitive precondition that the total number of equipartition and arrangement conditions cannot exceed the $kd$ degrees of freedom for $k$ hyperplanes in $\mathbb{R}^d$.

\begin{theorem} \label{thm:2.1} Let $C(\mathbf{m},\mathbf{a},\mathcal{O}; \mathbf{k})=\sum_{i=1}^k [m_i(2^{k-i+1}-1) + a_i] + |\mathcal{O}|$. Then
 \begin{equation}\label{eq:2.1} k\Delta(\mathbf{m},\mathbf{a},\mathcal{O}; k)\geq C(\mathbf{m},\mathbf{a},\mathcal{O}; \mathbf{k}). \end{equation} 
\end{theorem} 

\begin{proof}  As a preliminary observation, it follows from a standard compactness argument applied to measures concentrated at points (see, e.g., [18, 22]) that $\Delta(m;k)\leq d$ implies that for any $m$ point collections $C_1$,\ldots, $C_m$ in $\mathbb{R}^d$, there must exist some limiting collection of $k$ (possibly non-distinct) hyperplanes $H_1,\ldots, H_k$ such that the interior of each $\mathcal{R}_g= H_1^{g_1}\cap \cdots \cap H_k^{g_k}$ (some of which are possibly empty) contains at most $\frac{1}{2^k}$ of the points from each collection $C_i$. In particular, if $|C_i|=2^k-1$ for each $1\leq i \leq m$, then $\mathcal{R}_g\cap C_i=\emptyset$ for all $g\in \mathbb{Z}_2^{\oplus k}$ and all $1\leq i\leq k$, and therefore each of the $m(2^k-1)$ points must lie on the union of these hyperplanes. 

	Supposing that $\Delta(\mathbf{m},\mathbf{a},\mathcal{O}; k)\leq d$, let $M=\sum_{i=1}^k m_i(2^{k-i+1}-1)$ and let $a=\sum_{i=1}^k a_i$. We will show that any $D:=M+a$ points in $\mathbb{R}^d$ must lie on the union of (at most) $k$ hyperplanes, (at least) $|\mathcal{O}|$ of which are orthogonal, so that $kd-|\mathcal{O}|\geq D$ follows by generically choosing the $D$ points. To that end, for each $1\leq i \leq m$ consider a point collection $\{C_{i,j}\}_{1\leq j \leq m_i}$ with $|C_{i,j}|=2^{k-i+1}-1$. If $\mathbf{a}\neq \mathbf{0}$, let $A_1,\ldots, A_k$ be any flats such that $A_i$ is $(a_i-1)$-dimensional and such that $A_i$ contains the points $p_{i,1},\ldots, p_{i,a_i}$ from some collection $P=\{p_{i,j}\mid 1\leq i \leq k, 1\leq j \leq a_i\}$ disjoint from $\{C_{i,j}\}_{1\leq j \leq m_i}$. Concentrating measures at the $M$ points as before, it follows from compactness that at least $|\mathcal{O}|$ of the limiting $H_1,\ldots, H_k$ must be orthogonal, that $\{p_{i,j}\}_{j=1}^{a_i} \subset H_i$ for each $i$, and that $\cup_{\ell=1}^{k-i+1} H_\ell$ contains $\{C_{i,j}\}_{1\leq j \leq m_i}$ for each $i$ as well.  \end{proof}	

\section{Previous Equivariant Constructions for $\Delta(m;k)$} 

We present a detailed review of the Configuration-Space/Test-Map Scheme as applied to the classical Gr\"unabum--Ramos problem; the modifications required to incorporate our constraints are given in Section 4. 

\subsection{Configuration Spaces} The central observation for the introduction of equivariant topology to Question \ref{q:1} is the identification of each pair of complementary half-spaces $\{H^0,H^1\}$ in $\mathbb{R}^d$ with a unique pair  $\{\pm x\}\subset S^d$ of antipodal points on the unit sphere in $\mathbb{R}^{d+1}$:
\begin{equation} \label{eq:3.1} H^0(x): = \{\mathbf{u}\in \mathbb{R}^d\mid \langle \mathbf{u}, \mathbf{a} \rangle \geq b\} \,\,\, \text{and} \,\,\, H^1(x): = \{\mathbf{u}\in \mathbb{R}^d\mid \langle \mathbf{u}, \mathbf{a} \rangle \leq b\} = H^0(-x), \end{equation} 
\noindent where $x=(\mathbf{a},b)\in S^d\subset \mathbb{R}^d\times \mathbb{R}$. These sets are genuine half-spaces when $x\neq \pm (\mathbf{0},1)$, while $H^0(\mathbf{0},1)=\mathbb{R}^d$ and $H^1(\mathbf{0},1)=\emptyset$ correspond to a ``hyperplane at infinity". 

	Given this identification, all collections of regions $\{\mathcal{R}_g\}_{g\in \mathbb{Z}_2^{\oplus k}}$ determined by any $k$-tuple of hyperplanes (some possibly at infinity and not necessarily distinct)  are parametrizable by the $k$-fold product $(S^d)^k$: \begin{equation} \label{eq:3.2} \mathcal{R}_g(x):=\cap_{i=1}^k H^{g_i}(x_i)=\cap_{i=1}^k H^0((-1)^{g_i} x_i) \end{equation} for each $g\in \mathbb{Z}_2^{\oplus k}$ and each $x=(x_1,\ldots, x_k)\in (S^d)^k$. In addition to  the standard $\mathbb{Z}_2^{\oplus k}$-action on the product of spheres corresponding to that on $\{\mathcal{R}_g\}_{g\in \mathbb{Z}_2^{\oplus k}}$ by reflections about the hyperplanes, for $k>1$ one also has the action of the symmetric group on both $(S^d)^k$ and on any $k$-tuples of hyperplanes, and therefore that of the wreath product 
\begin{equation} \label{eq:3.3} \mathfrak{S}_k^\pm=\{g\rtimes \sigma\mid g\in \mathbb{Z}_2^{\oplus k}\,\, \text{and}\,\, \sigma \in \mathfrak{S}_k\}. \end{equation} Explicitly, $\mathfrak{S}_k$ acts on $\mathbb{Z}_2^{\oplus k}$ by $\sigma\cdot g:=(g_{\sigma^{-1}(1)},\ldots, g_{\sigma^{-1}(k)})$, resulting in the standard action  
	\begin{equation} \label{eq:3.4} (g \rtimes \sigma)\cdot (x_1,\ldots, x_k)= ((-1)^{g_1} x_{\sigma^{-1}(1)},\ldots, (-1)^{g_k}x_{\sigma^{-1}(k)}) \end{equation} on $(S^d)^k$ corresponding to that of $\mathfrak{S}_k^{\pm}$ on each collection $\{\mathcal{R}_h(x)\}_{h\in \mathbb{Z}_2^{\oplus k}}$:  
 \begin{equation} \label{eq:3.5} (g\rtimes \sigma)\cdot \mathcal{R}_h(x_1,\ldots, x_k):=\mathcal{R}_g\left((-1)^{h_{\sigma^{-1}(1)}}x_{\sigma^{-1}(1)}, \ldots, (-1)^{h_{\sigma^{-1}(k)}}x_{\sigma^{-1}(k)}\right). \end{equation}   

In addition to the product scheme primarily used (see, e.g., [6, 16, 23]) one can also consider as in [5, 8] the $k$-fold join 
	\begin{equation} \label{eq:3.6} (S^d)^{\star k}=\{\lambda\,x:=\sum_{i=1}^k \lambda_ix_i\mid x\in (S^d)^k,\, 0\leq \lambda_i\leq 1, 1\leq i \leq k,\, \text{and}\, \sum_{i=1}^k\lambda_i=1\} \end{equation}  consisting of all formal convex combination of the $x_i$. Thus $(S^d)^{\star k}$ is a topological sphere of dimension $k(d+1)-1$. The action on the join is defined similarly as before, but requires the slight modification that $(g\rtimes \sigma)\cdot \sum_{i=1}^k \lambda_i x_i = \sum_{i=1}^k  \lambda_{\sigma^{-1}(i)} (-1)^{g_i} x_{\sigma^{-1}(i)}$. In particular, note that the product can be seen inside the join via the diagonal embedding  $x\mapsto \sum_{i=1}^k \frac{1}{k}\, x_i$.

\subsection{Target Spaces and Test Maps} Evaluating the difference between the measure of a region and $\frac{1}{2^k}$ of the total mass defines a $\mathfrak{S}_k^\pm$-equivariant map, a zero of which represents the desired equipartition. Namely, $\mathfrak{S}^\pm_k$-acts on the regular representation $\mathbb{R}[\mathbb{Z}_2^{\oplus k}]$ by  $(g\rtimes \sigma)\cdot \left(\sum_{h\in \mathbb{Z}_2^{\oplus k}} r_h\, h\right) = \sum_{h\in \mathbb{Z}_2^{\oplus k}} r_{g + \sigma \cdot h} h$, and one considers the $(2^k-1)$-dimensional subrepresentation \begin{equation} \label{eq:3.7} U_k:=\left \{\sum_{h\in \mathbb{Z}_2^{\oplus k}} r_h h \mid r_h\in \mathbb{R}\, \,\, \text{and} \,\, \sum_h r_h=0 \right\}, \end{equation} i.e., the orthogonal complement of the (diagonal) trivial representation inside $\mathbb{R}[\mathbb{Z}_2^{\oplus k}]$.  Given any collection $\mathcal{M}=\{\mu_1,\ldots, \mu_m\}$ of $m$ masses on $\mathbb{R}^d$, evaluating measures in the product scheme produces a continuous map 
\begin{equation} \label{eq:3.8} \phi_{\mathcal{M}}=(\phi_1,\ldots, \phi_m): (S^d)^k\rightarrow U_k^{\oplus m} \end{equation} defined by 
\begin{equation} \label{eq:3.9} \phi_i(x)=\sum_{h \in \mathbb{Z}_2^{\oplus k}} \left(\mu_i(\mathcal{R}_h(x))- \frac{1}{2^k}\mu_i(\mathbb{R}^d)\right) h\end{equation} for each $1\leq i \leq m$, and this map is $\mathfrak{S}_k^\pm$-equivariant with respect to the above actions by construction. For the join, one lets $W_k=\{(t_1,\ldots, t_k)\in\mathbb{R}^k \mid \sum_{i=1}^k t_i=0\},$ considered as a $\mathfrak{S}_k^\pm$-representation under the action $(g\rtimes \sigma) \cdot (t_1,\ldots, t_k)=(t_{\sigma^{-1}(1)},\ldots, t_{\sigma^{-1}(k)})$, and defines a $\mathfrak{S}_k^{\pm}$-equivariant map 
					\begin{equation} \label{eq:3.10} \Phi_{\mathcal{M}}: (S^d)^{\star k} \rightarrow U_k^{\oplus m} \oplus W_k \end{equation} 
\noindent by \begin{equation} \label{eq:3.11} \Phi_{\mathcal{M}}(\lambda\, x) = (\lambda_1\cdots \lambda_k) \phi_{\mathcal{M}}(x)\oplus \left(\lambda_1-\frac{1}{k},\ldots, \lambda_k-\frac{1}{k}\right). \end{equation}

 Note that $\Phi_{\mathcal{M}}$ is well-defined, and it is immediately verified that a zero of either $\Phi_{\mathcal{M}}$ or $\phi_{\mathcal{M}}$ corresponds to $k$ genuine and distinct hyperplanes, and hence that $\Delta(m;k)\leq d$ if either of these maps vanishes. Indeed, (a) the zeros  (if any) of the test-map $\Phi_{\mathcal{M}}$ are the diagonal embedding of those of $\phi_{\mathcal{M}}$, (b) if $x_i= \pm(\mathbf{0},1)$ for any $1\leq i \leq k$ then at least one of the half-spaces would be empty, so that at least one of the $\mathcal{R}_g(x)$ would have measure zero, and (c) if $x_i=\pm x_j$ for some $1\leq i \leq j<k$, then one of the corresponding regions would be empty. 
 
 Before restating the topological upper bounds on $\Delta(m;k)$ obtained by these constructions, we make a few preliminary comments. First, the action on either configuration space is not free when $k\geq 2$: the isotropy groups $\mathfrak{S}_k^\pm(x)$ are non-trivial iff $x_i=\pm x_j$ for some $1\leq i < j\leq k$, and likewise $\mathfrak{S}_k^\pm(\lambda\,x) \neq \{e\}$ iff (a) $x_i=\pm x_j$ and $\lambda_i=\lambda_j$ for some $1\leq i<j\leq k$, or (b) when $\lambda_i=0$ for some $1\leq i \leq k$. None of the points from the singular sets \begin{equation} \label{eq:3.12} (S^d)^k_1=\{x\in (S^d)^k\mid \mathfrak{S}_k^\pm(x) \neq \{e\}\} \,\, \text{or} \,\, (S^d)^{\star k}_1=\{\lambda\,x \in (S^d)^{\star k}\mid \mathfrak{S}_k^\pm(\lambda\,x) \neq \{e\}\} \end{equation} are zeros of the maps $\Phi_{\mathcal{M}}$ or $\phi_{\mathcal{M}}$, however. Second, while we shall not use it here other than in the statement of Theorem 3.1 below, it is easily seen as in [5, Proposition 2.2] that if $f: (S^d)^k \stackrel{\mathfrak{S}_k^\pm}{\rightarrow} U_k^{\oplus m}$ is \textit{any} map which is equivariant with respect to the actions described above and whose zeros are guaranteed to lie outside $(S^d)^k_1$, then the restrictions $\phi_{\mathcal{M}}|(S^d)^k_1  \stackrel{\mathfrak{S}_k^\pm}{\simeq}(S^d)^k_1$ are equivariantly and linearly homotopic, and so $F|(S^d)^{\star k}_1  \stackrel{\mathfrak{S}_k^\pm}{\simeq}\Phi_{\mathcal{M}}|(S^d)^{\star k}_1$ as well, where $F: (S^d)^{\star k} \rightarrow U_k^{\oplus m}\oplus W_k$ is defined by $F(\lambda\, x) = (\lambda_1\cdots \lambda_k) f(x)\oplus (\lambda_1-\frac{1}{k},\ldots, \lambda_k-\frac{1}{k})$.  In particular, this has the consequence that the resulting obstruction classes are  independent of the masses considered. Finally, while the product scheme may seem  to be the more natural construction, and although the zeros of $\Phi_{\mathcal{M}}$ and $\phi_{\mathcal{M}}$ are in bijective correspondence, the join scheme has the computational advantage that the resulting configuration space is a sphere and hence (by contrast with the product) is connected up to top dimension and therefore more easily suited to (relative) equivariant obstruction theory.

\begin{theorem} \label{thm:3.1} [\emph{[5, Theorem 2.3]}]     
 Let $f: (S^d)^k \stackrel{\mathfrak{S}_k^\pm}{\rightarrow} U_k^{\oplus m}$ be such that $f(x)\neq 0$ for all $x\in (S^d)^k_1$, and let $F: (S^d)^{\star k} \stackrel{\mathfrak{S}_k^\pm}{\rightarrow} U_k^{\oplus m}\oplus W_k$ be given by $F(\lambda\, x) = (\lambda_1\cdots \lambda_k) f(x)\oplus (\lambda_1-\frac{1}{k},\ldots, \lambda_k-\frac{1}{k})$. 
\\
(a) If $f:(S^d)^k \stackrel{\mathfrak{S}_k^\pm}{\rightarrow} U_k^{\oplus m}$ has a zero, then $\Delta(m;k)\leq d$.\\
(b) Let $H: (S^d)^{\star k} \stackrel{\mathfrak{S}_k^\pm}{\rightarrow} U_k^{\oplus m} \oplus W_k$, and suppose that $H|(S^d)^{\star k}_1 \stackrel{\mathfrak{S}_k^\pm}{\simeq}F|(S^d)^{\star k}_1$. If $H$ has a zero, then $\Delta(m;k)\leq d$. 
\end{theorem}

 \noindent \textbf{Remark 1.} All presently known upper bounds for $\Delta(m;k)$ can be seen as either immediate consequences of Theorem \ref{thm:3.1}, or of Theorem \ref{thm:3.1} combined with the reduction  $\Delta(m;k+1)\leq \Delta(2m;k)$ obtained by simultaneously bisecting any given $m$ measures in $\mathbb{R}^{\Delta(2m;k)}$ by a single hyperplane. The general upper bound $U(m;k)$  -- including $\Delta(2^{q+1}-1;2)=3\cdot 2^q -1$ for all $q\geq0$  -- follows from case (a) by restricting to $\mathbb{Z}_2^{\oplus k}$-equivariance [16, Theorem 4.2], $\Delta(2^{q+2}+1;2)=3\cdot 2^q +1$ for all $q\geq 0$ also follows from (a) by using the full wreath product action [6, Theorem 5.1], and case (b) yields $\Delta(2^{q+1};2)=3\cdot 2^q$ for all $q\geq0$, $\Delta(2;3)=5$, $\Delta(4;3)=10$ [5, Theorems 1.5--1.6]. In fact, (b) also recovers $\Delta(2^{q+1}-1;2)=3\cdot 2^q -1$ and $\Delta(2^{q+2}+1;2)=3\cdot 2^q +1$ via [5, Theorem 1.4]. The remaining known value $\Delta(1;3)=3$ follows by reduction, as do the estimates $4\leq \Delta(1;4)\leq 5$ and $8\leq \Delta(2;4)\leq 10$.

\section{Constraints as Enlarged Representations} Constrained variants on the equipartition problem can be introduced simply into the above scheme  by enlarging the target spaces and test maps. When the resulting extension is a subrepresentation of $U_k^{\oplus n}$, $n>m$, the full power of Theorem \ref{thm:3.1} can be used. While this is the case for $\Delta^\perp(m;k)$, those arising from Question 2 more generally are rarely invariant under the full $\mathfrak{S}_k^\pm$-action (but are always invariant under the action of the $\mathbb{Z}_2^{\oplus k}$-subgroup), so our results in those circumstances follow instead from a Borsuk--Ulam type theorem for arbitrary $\mathbb{Z}_2^{\oplus k}$-modules (Proposition \ref{prop:6.1}). 

	In either case, we recall that the irreducible representations of $\mathbb{Z}_2^{\oplus k}$ are all 1-dimensional and are indexed by the group itself, with each $\alpha=(\alpha_1,\ldots, \alpha_k) \in \mathbb{Z}_2^{\oplus k}$ giving rise to $\rho_\alpha: \mathbb{Z}_2^{\oplus k} \rightarrow \mathbb{R}$ defined by $\rho_\alpha(g)=(-1)^{\alpha_1g_1 +\cdots+  \alpha_k g_k}$. Given $A\subseteq \mathbb{Z}_2^{\oplus k}$, the $\mathbb{Z}_2^{\oplus k}$-representation $V=\oplus_{\alpha\in A} V_\alpha$ is a $\mathfrak{S}_k^\pm$-subrepresentation of $\mathbb{R}[\mathbb{Z}_2^{\oplus k}]$ iff  is closed under the action of the symmetric group on $\mathbb{Z}_2^{\oplus k}$ described in Section 3, as is the case with  $U_k\cong \oplus_{\alpha\neq 0} V_\alpha$. With this viewpoint, we can now describe the CS/TM set-up for the various constraints of Problem 2.

\subsection{Orthogonality} Given $\mathcal{O}\subseteq \{(r,s)\mid 1\leq r<s\leq k\}$, let $\mathbf{e}_1,\ldots, \mathbf{e}_k$ denote the standard basis vectors for $\mathbb{Z}_2^{\oplus k}$ and let \begin{equation} \label{eq:4.1} A(\mathcal{O})=\{(\alpha_1,\ldots, \alpha_k)\mid \alpha_r=\alpha_s=1\,\, \text{and}\,\, \alpha_i=0\,\, \text{if}\,\, i\neq r,s\}. \end{equation} Thus
\begin{equation} \label{eq:4.2} V_{A(\mathcal{O})}=\oplus_{(r,s)\in\mathcal{O}} V_{\mathbf{e}_r+\mathbf{e}_s}. \end{equation}
Letting $q: S^d\rightarrow \mathbb{R}^d$ denote the projection of each $x\in S^d$ onto its first $d$ coordinates, evaluating the various inner products $\langle q(x_r),q(x_s)\rangle$ for each $x=(x_1,\ldots, x_k) \in (S^d)^k$ and each $(r,s)\in \mathcal{O}$ yields a $\mathbb{Z}_2^{\oplus k}$-equivariant map 
\begin{equation} \label{eq:4.3} \phi_{\mathcal{O}}: (S^d)^k \rightarrow V_{A(\mathcal{O})},\end{equation} and thus one has the desired orthogonal equipartition if $\phi_{\mathcal{M}}\oplus \phi_{\mathcal{O}}: (S^d)^k \rightarrow  U_k^{\oplus m} \oplus V_{A(\mathcal{O})}$ has a zero. This map is $\mathfrak{S}_k^\pm$-equivariant iff $A(\mathcal{O})$ is $\mathfrak{S}_k$-invariant, as is the case when full orthogonality is prescribed. One has the analogous construction for the join, with $\phi_{\mathcal{O}}$ replaced by \begin{equation} \label{eq:4.4} \Phi_{\mathcal{O}}: (S^d)^{\star k}\rightarrow V_{A(\mathcal{O})}\end{equation} defined by $\Phi_{\mathcal{O}}(\lambda\,x)=(\lambda_1\cdots \lambda_k)\phi_{\mathcal{O}}(x)$.

\subsection{Affine Containment} Prescribing that each $H_i$ contains a given $(a_i-1)$-dimensional affine subspace $A_i$ can be seen by adjoining $\oplus_{i=1}^k V_{\mathbf{e}_i}^{a_i}$ to $U_k^{\oplus m}$, with similar remarks for the join scheme. Namely, $H_i$ contains a given point $p$ iff $H_i$ equipartitions the mass $\mu_p$ defined by a unit ball centered at $p$, so $A_i=\text{aff}(p_{i,1},\ldots, p_{i,k})\subseteq H_i$ iff $H_i$ simultaneously equipartitions the $\mu_{p_{i,1}},\ldots, \mu_{p_{i,a_i}}$. Clearly, \begin{equation} \label{eq:4.5} \phi_{\mathcal{A}}:= (S^d)^k  \rightarrow \oplus_{i=1}^k V_{\mathbf{e}_i}^{a_i} \end{equation} given by $x\mapsto [\mu_{p_{i,j}}(H_i^0(x_i))-\frac{1}{2}\mu_{p_{i,j}}(\mathbb{R}^d)]$ for each $1\leq i \leq k$ and each $1\leq j \leq a_i$ is  $\mathbb{Z}_2^{\oplus k}$-equivariant in general, and $\mathfrak{S}_k^\pm$-equivariant iff $a_1=\cdots = a_k$.

\subsection{Cascades} Let $\mathcal{M}_1=\{\mu_{1,1},\ldots, \mu_{1,m_1}\},\ldots, \mathcal{M}_k=\{\mu_{k,1},\ldots, \mu_{k,m_k}\}$ be $k$ collections of masses on $\mathbb{R}^d$. Recall that a ``cascade" by hyperplanes $H_1,\ldots, H_k$ means that $H_i,\ldots, H_k$ equipartitions $\mathcal{M}_i$ for each $1\leq i \leq k$. 

	For any $1\leq i \leq k$, let $\pi_i: (S^d)^k\rightarrow (S^d)^{(k-i+1)}$ denote the projection onto the last $(k-i+1)$ coordinates. Each $\pi_i(x)$ then determines a $(k-i+1)$-tuple of hyperplanes with corresponding regions $\mathcal{R}_{(g_{k-i+1},\ldots, g_k)}(\pi_i(x))=\cap_{\ell=k-i+1}^k H^{g_\ell}(x_\ell)$ for each $(g_{k-i+1},\ldots, g_k) \in \mathbb{Z}_2^{\oplus(k-i+1)}$. Letting $U_{k,i}$ denote the orthogonal complement of the trivial representation inside $\mathbb{R}[\mathbb{Z}_2^{\oplus (k-i+1)}]$ as before, define \begin{equation} \label{eq:4.6} \phi_{\mathcal{M}_i}:= (S^d)^k \rightarrow U_{k,i}^{\oplus m_i} \end{equation} by \begin{equation} \label{eq:4.7} x\mapsto \sum_{(g_{k-i +1}, \ldots, g_k)\in \mathbb{Z}_2^{\oplus (k-i+1)}}\left(\mu_{i,j}(\mathcal{R}_{(g_{k-i+1},\ldots, g_k)}(\pi_i(x)) -\frac{1}{2^{k-i+1}}\mu_{i,j}(\mathbb{R}^d)\right)(g_{k-i+1},\ldots, g_k) \end{equation} for each $1\leq j\leq m_i$. It is easily verified that this map is $\mathfrak{S}_k^{\pm}$-equivariant, where the group acts standardly on $(S^d)^k$, and on $U_{k,i}$ by projecting $\mathfrak{S}_k^{\pm}$ onto $\mathfrak{S}_{k-i+1}^\pm$ and letting the latter act as before. The desired cascade is thus a zero of the $\mathfrak{S}_k^\pm$-equivariant map \begin{equation} \label{eq:4.8} \phi_{\mathcal{M}_\mathbf{m}}=\phi_{\mathcal{M}_1}\oplus \cdots \oplus \phi_{\mathcal{M}_k}: (S^d)^k \rightarrow U_{k,1}^{\oplus m_1}\oplus \cdots \oplus U_{k,k}^{\oplus m_k} \end{equation} 

	Note that $\oplus_{i=1}^k U_{k,1}^{\oplus m_i}$ is not a $\mathfrak{S}_k^\pm$-subrepresentation of $U_k^{\oplus (m_1+\cdots m_k)}$ under the action considered in Section 3, nor is the map $\phi_{\mathcal{M}_\textbf{m}}$ equivariant with respect to those actions unless $\mathbf{m}= m_1\mathbf{e}_1$ (in which case the cascade condition is vacuous). Again, similar remarks hold for the join scheme. 

\section{Full Orthogonality via Restriction} 

Having given the necessary changes in the CS/TM scheme for the constrained cases, the remainder of this paper presents topological upper bounds to Question \ref{q:2}. 

	We begin with $\Delta^\perp(m;k)$, in which case an observation of Florian Frick shows that the restriction trick [6, Proposition 3.3] of Blagojevi\'c, Frick, Hasse, and Ziegler for certain $\mathfrak{S}_k^\pm$-subrepresentations of $U_k^{\oplus (m+1)}$ easily produces upper bounds for full orthogonality from those for $\Delta(m+1;k)$ obtained equivariantly. To that end, let $\mathcal{B}=\{\mathbf{e}_1,\ldots, \mathbf{e}_k\}$ denote the standard basis of $\mathbb{Z}_2^{\oplus k}$, and let $\mathcal{B}^C$ denote its complement inside $\mathbb{Z}_2^{\oplus k}$. 

\begin{lemma} \label{lem:5.1} [\emph{[6, Proposition 3.3]}]   Let $f: (S^d)^k \stackrel{\mathfrak{S}_k^\pm}{\longrightarrow} U_k^{\oplus m}$ and $F: (S^d)^{\star k} \stackrel{\mathfrak{S}_k^\pm}{\longrightarrow} U_k^{\oplus m}\oplus W_k$ be as in Theorem 3.1, and suppose that $f': (S^{d-1})^k \stackrel{\mathfrak{S}_k^\pm} \rightarrow U_k^{\oplus (m-1)}\oplus_{\alpha\in \mathcal{B}^C} V_\alpha$ is without zeros on $(S^{d-1})^k_1$ and that $F': (S^{d-1})^{\star k} \stackrel{\mathfrak{S}_k^\pm} \rightarrow U_k^{\oplus (m-1)} \oplus_{\alpha\in \mathcal{B}^C} V_\alpha \oplus W_k$ is defined by $F'(\lambda\, x) = (\lambda_1\cdots \lambda_k) f'(x)\oplus (\lambda_1-\frac{1}{k},\ldots, \lambda_k-\frac{1}{k}).$\\
(a) If any such $f: (S^d)^k \stackrel{\mathfrak{S}_k^\pm}{\rightarrow} U_k^{\oplus m}$ vanishes, then so does $f': (S^{d-1})^k\stackrel{\mathfrak{S}_k^\pm} \rightarrow U_k^{\oplus (m-1)}\oplus_{\alpha\in \mathcal{B}^C} V_\alpha $.\\
(b) If any $H: (S^d)^{\star k} \stackrel{\mathfrak{S}_k^\pm}{\rightarrow} U_k^{\oplus m} \oplus W_k$ such that $H|(S^d)^{\star k}_1 \stackrel{\mathfrak{S}_k^\pm}{\simeq}F|(S^d)^{\star k}_1$ vanishes, then so does  any $H': (S^{d-1})^{\star k} \stackrel{\mathfrak{S}_k^\pm}{\rightarrow} U_k^{\oplus m} \oplus_{\alpha\in \mathcal{B}^C} V_\alpha \oplus W_k$ such that $H'|(S^{d-1})^{\star k}_1 \stackrel{\mathfrak{S}_k^\pm}{\simeq}F'|(S^{d-1})^{\star k}_1$.\\
\end{lemma} 

\begin{proof} As part (b) is absent from [6],  we include it here for the sake of completeness even though its proof is essentially the same as that of (a) given there. Let $Q: (S^d)^{\star k} \rightarrow \oplus _{i=1}^k V_{\mathbf{e}_i}$ be the $\mathfrak{S}_k^\pm$-map sending each $\sum_{i=1}^k\lambda_ix_i$ to $(\lambda_1q_k(x_1),\ldots, \lambda_kq_k(x_k))$, where $q_k(x_i)$ is the last coordinate of $x_i\in S^d$. Clearly, $Q^{-1}(0)=(S^{d-1})^{\star k}$. Viewing $S^d$ as the join of $S^{d-1}$ and $S^0$, let $H: (S^d)^{\star k}=(S^{d-1}\star S^0)^{\star k} \stackrel{\mathfrak{S}_k^\pm}{\rightarrow} U_k^{\oplus (m-1)}  \oplus_{\alpha\in \mathcal{B}^C} V_\alpha \oplus W_k$ be the extension of the given $H': (S^{d-1})^{\star k} \rightarrow  U_k^{\oplus (m-1)} \oplus_{\alpha\in \mathcal{B}^C} V_\alpha \oplus W_k$ defined by $ H\left(\sum_{i=1}^k\lambda_i(t_ix_i+(1-t_i)y_i)\right) = (t_1\cdots t_k)H'\left(\sum_{i=1}^k\lambda_ix_i\right)$,  where $x_i\in S^{d-1}$, $y_i\in S^0$, and $0\leq t_i\leq 1$. Letting $F$ be the analogous extension of $F'$, one has maps $(H\oplus Q), (F\oplus Q): (S^d)^{\star k} \stackrel{\mathfrak{S}_k^\pm}{\rightarrow} U_k^{\oplus m} \oplus W_k$, and it follows from $H'|(S^{d-1})^{\star k}_1 \stackrel{\mathfrak{S}_k^\pm}{\simeq}F'|(S^{d-1})^{\star k}_1$ that $(H\oplus Q)|(S^d)^{\star k}_1 \stackrel{\mathfrak{S}_k^\pm}{\simeq}(F\oplus Q)|(S^d)^{\star k}_1$ as well. However, any zero of $H\oplus Q$ map must lie in $(S^{d-1})^{\star k}$ and hence must be a zero of the original map $H'$.\end{proof}

\begin{theorem} \label{thm:5.2} If $\Delta(m;k)\leq d$ via Theorem 3.1, then $\Delta^\perp(m-1;k)\leq d-1$. Thus for $q\geq 0$ one has \\
$\bullet$ $\Delta^\perp(2^{q+1}-1;2)=3\cdot 2^q -1$\\
$\bullet$ $\Delta^\perp(2^{q+2}-2;2)=3\cdot 2^{q+1}-2$ \\
$\bullet$ $\Delta^\perp(2^{q+2}; 2)=3\cdot 2^{q+1}+1$\\
$\bullet$ $\Delta^\perp(1;3)=4$ and $8\leq \Delta^\perp(3;3)\leq 9$\\
$\bullet$ $\Delta^\perp(2^q+r;k)\leq U(2^q+r;k):=2^{q+k-1}+r$ for all $0\leq r\leq 2^q-2$
\end{theorem}

\begin{proof} As $V_{A({\mathcal{O}})}$ is a $\mathfrak{S}_k^\pm$-subrepresentation of $V_{\mathcal{B}^C}$ when $\mathcal{O}=\{(r,s)\mid 1\leq r<s\leq k\}$, this an immediate consequence of Theorem 3.1, Remark 1 above, and Lemma 5.1 applied to (a) $\phi_{\mathcal{M}}\oplus \phi_{\mathcal{O}}: (S^d)^k\rightarrow  U_k^{\oplus (m-1)} \oplus V_{A(\mathcal{O})}$ or  (b) $\Phi_{\mathcal{M}}\oplus \Phi_{\mathcal{O}}: (S^d)^{\star k} \rightarrow  U_k^{\oplus (m-1)} \oplus W_k \oplus V_{A(\mathcal{O})}$, where $\phi_{\mathcal{O}}$ and $\Phi_{\mathcal{O}}$ are from Section 4.1. \end{proof}

\noindent \textbf{Remark 2.} Theorem \label{thm:5.2} shows in particular that the best-known general upper bound for $\Delta(m;k)$ also holds in the orthogonal cases whenever $m+1$ is not a power of 2. As with Tverberg-type problems via [4], it is a testament to the power of restriction that while $\Delta^\perp(m;k)$ has been previously considered in the literature (see, e.g., [7, 18]), the above theorem gives the first exact values other than the very classical $\Delta^\perp(1;2)=2$ which can be found in [9]. Note that $\Delta^\perp(2^{q+1}-1;2)=\Delta(2^{q+1}-1;2)$, so that the one remaining degree of freedom between $k\Delta(m;k)$ and $m(2^k-1)$ in the Ramos lower bound (1.2) has been used  (see also Corollary \ref{cor:7.1} below, where orthogonality is swapped for bisection of any further prescribed measure by one of the hyperplanes). On the other hand, there is still one remaining degree of freedom in the constrained lower bound \eqref{eq:2.1} for the other cases of $\Delta^\perp(m;2)$ in Theorem \ref{thm:5.2}, and these are inaccessible via Lemma \ref{lem:5.1}. Moreover, when $k\geq 3$ there are $(2^k-1-\binom{k}{2})$ degrees of freedom remaining in the domain of Lemma \ref{lem:5.1}, and these cannot be used for either cascades or affine containment since those require some $V_{\mathbf{e}_i}$ in the test-space. Nonetheless, in the following two sections we show that $\Delta^\perp((2^{q+2}-2,1);2)=3\cdot 2^{q+1}-2$ (Corollary \ref{cor:7.1}), and in Theorem \ref{thm:6.4} we recover $\Delta^\perp(m;k)\leq U(m;k)$ above while including cascades and affine containment constraints which remove all remaining degrees of freedom in \eqref{eq:2.1}. 

\section{Optimizing the Topological Upper Bound $U(m;k)$} 

\subsection{Cohomological Preliminaries}

We now turn to upper bounds on $\Delta(\mathbf{m};\mathbf{a};\mathcal{O})$ more generally. Considering the product scheme, the resulting test-maps are
\begin{equation} \label{eq:6.1} \phi_{\mathcal{M_{\mathbf{m}}}}\oplus \phi_{\mathcal{A}}\oplus \phi_{\mathcal{O}}: (S^d)^k \rightarrow \oplus_{i=1}^k U_{k,i}^{\oplus m_i}  \oplus_{i=1}^{a_i} V_{\mathbf{e}_i}^{a_i} \oplus V_{A(\mathcal{O})}. \end{equation} 
As we saw in Section 4, the resulting target space is not a $\mathfrak{S}_k^\pm$-subrepresentation of $U_k^{\oplus n}$ (nor of $U_k^{\oplus n}\oplus W_k$ in the join construction) unless $\mathbf{m}=m\mathbf{e}_1$, $\mathbf{a}=a(\mathbf{e}_1+\cdots \mathbf{e}_k)$, and $A(\mathcal{O})$ is a symmetric subset of $\mathbb{Z}_2^{\oplus k}$. Thus the full wreath product action is not well-suited to address Question \ref{q:2} except in special circumstances, and therefore the full strength of either parts (a) or (b) of Theorem \ref{thm:3.1} are unavailable. All of these maps are certainly equivariant when restricted to the $\mathbb{Z}_2^{\oplus k}$-subgroup, however, so that results to Question \ref{q:2} in dimension $U(m;k)$ for which equality holds in \eqref{eq:2.1} will be obtained from careful polynomial calculations using the following cohomological condition. 

\begin{proposition} \label{prop:6.1} Let $M=(\alpha_{i,j})_{1\leq i\leq kd,\,1\leq j\leq k}$ be a $(kd\times k)$-matrix with $\mathbb{Z}_2$-coefficients and let $h(u_1,\ldots, u_k)=\Pi_{i=1}^{kd} (a_{i,1}u_1+\cdots +a_{i,k}u_k)\in \mathbb{Z}_2[u_1,\ldots, u_k]/(u_1^{d+1},\ldots, u_k^{d+1})$. If $h(u_1,\ldots, u_k) = u_1^d\cdots u_k^d$, then any $\mathbb{Z}_2^{\oplus k}$-equivariant map $f: (S^d)^k \rightarrow \oplus_{i=1}^{kd} V_{(\alpha_{i,1},\ldots, \alpha_{i,k})}$ has a zero. 
\end{proposition} 

\begin{proof} While the proof below has an equivalent formulation in terms of the ideal-valued index theory of [11] commonly used in topological combinatorics, we shall use Steifel-Whitney classes instead in order to emphasize the remaining power of fundamental results in algebraic topology to mass partitions problems, here the $\mathbb{Z}_2$-cohomology of real projective space. For an introduction to the theory of vector bundles and characteristic classes as used below, see, e.g., the standard references [14, 17]. 

	Let $V=\oplus_{i=1}^{kd}V_{(\alpha_{i,1},\ldots, \alpha_{i,k})}$. If $f$ were non-vanishing, the section $x\mapsto (x,f(x))$ of the trivial bundle $(S^d)^k\times V$ would induce a non-vanishing section of the $kd$-dimensional real vector bundle $\xi: V\hookrightarrow (S^d)^k\times_{\mathbb{Z}_2^{\oplus k}} V\rightarrow (\mathbb{R}P^d)^k$ obtained by quotienting via the diagonal action. As such, the top Stiefel-Whitney class $w_{kd}(\xi)\in H^{kd}((\mathbb{R}P^d)^k;\mathbb{Z}_2)$ would be zero. It is an elementary fact that $H^*(\mathbb{R}P^d;\mathbb{Z}_2)=\mathbb{Z}_2[u]/(u^{d+1})$, where $u:=w_1(\gamma)$ is the first Stiefel-Whitney class of the canonical line bundle $\gamma:\mathbb{R}\hookrightarrow S^d\times_{\mathbb{Z}_2}\mathbb{R}\rightarrow \mathbb{R}P^d$, so that by the K\"unneth formula $H^*((\mathbb{R}P^d)^k;\mathbb{Z}_2)=\mathbb{Z}[u_1,\ldots, u_k]/(u_1^{d+1},\ldots, u_k^{d+1})$, where each $u_i:=\pi_i^*(u)=w_1(\gamma_i)$ is the first Stiefel-Whitney class of the pull-back bundle $\gamma_i:=\pi_i^*(\gamma)$ under the $i$-th factor projection $\pi_i: (\mathbb{R}P^d)^k\rightarrow \mathbb{R}P^d$. Thus $H^{kd}((\mathbb{R}P^d)^k;\mathbb{Z}_2)\cong\mathbb{Z}_2$ is generated by $u_1^d\cdots u_k^d$. As $\xi=\oplus_{i=1}^{kd}\xi_{(\alpha_{i,1},\ldots, \alpha_{i,k})}$ is the direct sum of the $\xi_{(\alpha_{i,1},\ldots, \alpha_{i,k})}: V_\alpha \hookrightarrow (S^d)^k \times _{\mathbb{Z}_2^{\oplus k}} V_{\alpha_i} \rightarrow (\mathbb{R}P^d)^k$, it follows from the Whitney Sum formula that $w_{kd}(\xi) = \Pi_{i=1}^{kd} w_1(\xi_{(\alpha_{i,1},\ldots, \alpha_{i,k})})$. Finally, each $\xi_{(\alpha_{i,1},\ldots, \alpha_{i,k})}=\gamma_1^{\alpha_{i,1}}\otimes \cdots \otimes \gamma_k^{\alpha_{i,k}}$ is the tensor product of the $\gamma_j$, so $w_1(\xi_{(\alpha_{i,1},\ldots, \alpha_{i,k})})=\alpha_{i,1}u_1+\cdots +\alpha_{i,k}u_k$ because $w_1(\chi_1\otimes \chi_2)=w_1(\chi_1)+w_1(\chi_2)$ for any two real line bundles.  Thus $w_{kd}(\xi)$ is precisely the polynomial $h(u_1,\ldots, u_k)$, hence $w_{kd}(\xi)=x_1^d\cdots x^d\neq 0$ and consequently $f: (S^d)^k \rightarrow V$ must have a zero. 
\end{proof}

We now give explicit formulae for the polynomials corresponding to the representations arising from conditions (i)---(iii) of Question \ref{q:2}.

\subsubsection{Orthogonality} Any $\mathcal{O}\subseteq \{(r,s)\mid 1\leq r<s\leq k\}$ gives rise to the $\mathbb{Z}_2^{\oplus k}$-representation $V_{A(\mathcal{O})}=\oplus_{(r,s)\in\mathcal{O}} V_{\mathbf{e}_r+\mathbf{e}_s}$, with resulting  polynomial \begin{equation} \label{eq:6.2} P_k(\mathcal{O}):=\Pi_{(r,s)\in \mathcal{O}}(u_r+u_s). \end{equation}  In particular, if $\mathcal{O}_j=\{(r,s)\mid j\leq r<s\leq k\}$, $1\leq j \leq k-1$, then $P_k(\mathcal{O}_j)$ is the Vandermonde determinant (see, e.g., [19]), and therefore \begin{equation} \label{eq:6.3} P_k(\mathcal{O}_j)=\sum_{\sigma \in \mathfrak{S}_{k-j+1}} u^{k-j}_{\sigma(j)}u^{k-j-1}_{\sigma(j+1)}\cdots u^0_{\sigma(k)}. \end{equation}

\subsubsection{Affine Containment} As the resulting representation is $\oplus_{i=1}^k V_{\mathbf{e}_i}^{\oplus a_i}$, the corresponding polynomial  is \begin{equation} \label{eq:6.4} P_k(\mathbf{a}):=u_1^{a_1}\cdots u_k^{a_k}. \end{equation}

\subsubsection{Cascades} As observed in [16, Theorem 4.1] for the proof of $\Delta(m;k)\leq U(m;k)$, $\Pi_{(\alpha_1,\ldots, \alpha_k)\neq 0}(x_{\alpha_1}+\cdots +x_{\alpha_k})$ arising from the $\mathbb{Z}_2^{\oplus k}$-representation $U_k$ is Dickson and can be expressed explicitly  (see, e.g., [21]) as $\sum_{\sigma \in \mathfrak{S}_k} u_{\sigma(1)}^{2^{k-1}}u_{\sigma(2)}^{2^{k-2}}\cdots u_{\sigma(k)}^1$. Thus each $U_{k,i}$ gives rise to the polynomial  \begin{equation} \label{eq:6.5}  P_{k,i}: =\sum_{\sigma\in \mathfrak{S}_{k-i+1}} u_{\sigma(i)}^{2^{k-i}}u_{\sigma(2)}^{2^{k-i-1}}\cdots u^1_{\sigma(k)}, \end{equation} and the polynomial corresponding to any cascading equipartition $\mathcal{M}_{\textbf{m}}$ is therefore \begin{equation} \label{eq:6.6} P_k(\mathbf{m}):=P_{k,1}^{m_1}\cdots P_k^{m_k}. \end{equation}

In what follows, we rewrite the upper bound $\Delta(m;k)=U(m;k)$ for $m=2^{q+1}-t$ by \begin{equation} \label{eq:6.7} U(2^{q+1}-t;k)=2^q\cdot [2^{k-1}+1]-t  \end{equation} \noindent for all $q\geq 0$ and $1\leq t\leq 2^q$.

\subsection{Cascading Equipartitions and Affine Containment} 

First, we consider Question \ref{q:2} in the absence of any orthogonality. We let $\Delta(\mathbf{m},\mathbf{a};k):=\Delta(\mathbf{m},\mathbf{a},\emptyset;k)$, and in particular denote the ``pure cascade" $\Delta(\mathbf{m},\mathbf{0};k)$ by $\Delta(\mathbf{m};k)$.

\begin{theorem} \label{thm:6.2} Let $1\leq t\leq 2^q$. If $\mathbf{a}=(a_1,\ldots, a_k)$ satisfies (i) $a_1\leq a_2\leq \cdots \leq a_k$,  (ii) $a_2\leq 2a_1+t$, and (iii) $0\leq a_{k-1}\leq 2^q-t$, then  \begin{equation} \label{eq:6.8} \Delta(\mathbf{m},\mathbf{a};k)=2^q \cdot [ 2^{k-1} + 1] -t \,\, \text{and}\,\, k\Delta(\mathbf{m},\mathbf{a};k)=C(\mathbf{m},\mathbf{a};k), \end{equation}
\noindent where $\mathbf{m}=(m_1,\ldots, m_k)$ is given by 
\begin{equation} \label{eq:6.9} m_1= 2^{q+1}-t-a_1\,\,\, and \,\,\, m_i=2^q\cdot [2^{i-2}-1]+t+2a_{i-1}-a_i  \,\,\, for\,\, all\,\, 2\leq i \leq k. \end{equation} 
\noindent Thus $\Delta(\mathbf{m},\mathbf{a};k)=U(m_1;k)$ whenever $a_1=0$, and in particular 	
		\begin{equation} \label{eq:6.10} \Delta((2^{q+1}-t,t,2^q+t, 3\cdot 2^q+t, \ldots, 2^q\cdot [2^{k-2}-1]+t);k)=U(2^{q+1}-t;k). \end{equation}  
\end{theorem}

 It is worth emphasizing that this theorem strengthens $\Delta(m_1;k)\leq U(m_1;k)$ for all $m_1$ and $k$ whenever $a_1=0$. Note also that $m_2\geq 0$ and $m_i\geq 1$ for all $i\neq 2$, with $m_2\geq 1$ provided $a_2< t+2a_1$.\\
\\
\textbf{Remark 3.} Considering affine containment, it follows in particular that (a) one may choose the (non-empty) affine subspaces $A_i$ to be linear, so that the $H_i$ are linear are as well, and moreover that (b) the condition $a_1\leq a_2\leq \cdots \leq a_k$ allows for filtrations $A_1\subseteq \cdots \subseteq A_k$, so that $(A_i\cap \cdots \cap A_k)\subseteq (H_i\cap \cdots \cap H_k)$ for all $1\leq i \leq k$.

\begin{proof} Let $d=2^q\cdot [2^{k-1}+1]-t$. By the discussion above, the relevant polynomial is \begin{equation} \label{eq:6.11} P:=P(u_1,\ldots, u_k)= P_{k,1}^{m_1}\cdot u_1^{a_1} \cdot P_{k,2}^{m_2}\cdot u_2^{a_2}\cdots P_{k,k}^{m_k}\cdot u_k^{a_k}.\end{equation} 
\noindent Defining  \begin{equation} \label{eq:6.12} h_i:=P_{k,i+1}^{m_{i+1}} \cdot u_{i+1}^{a_{i+1}} \cdots P_{k,k}^{m_k}\cdot u_k^{a_k} \end{equation} 

\noindent for all $0\leq i \leq k-1$ and $h_k:=1$, a simple induction argument will show that \begin{equation} \label{eq:6.13} P= u_1^d\cdots u_i^d \cdot  P_{k,i+1}^{2^{q+i}+2^q-t-a_{i+1}}\cdot u_{i+1}^{a_{i+1}} \cdot h_{i+1} \end{equation}  for all $0\leq i \leq k-1$, where it is to be understood that $u_1^d\cdot u_0^d=1$. Letting $i=k-1$ in \eqref{eq:6.13} then immediately yields $P=u_1^d\cdots u_k^d$. Clearly, \eqref{eq:6.13} holds when $i=0$. Assuming it true for $0\leq i\leq k-2$, consider $P_{k,i+1}^{2^{q+i} +2^q-t-a_{i+1}}=P_{k,i+1}^{2^{q+i}}\cdot  P_{k,i+1}^{2^q-t-a_{i+1}}$ and observe that  $P_{k,i+1}^{2^{q+i}}=\sum_{\sigma\in \mathfrak{S}_{k-i}} u_{\sigma(i+1)}^{2^{q+k-1}}\cdots u_{\sigma(k)}^{2^{q+i}}$. Letting $e(u_{\sigma(i+1)})$ denote  the exponent of $u_{\sigma(i+1)}$ in \eqref{eq:6.13} for each $\sigma\in \mathfrak{S}_{k-i}$, one has $e(u_{\sigma(i+1)})\geq 2^q\cdot 2^{k-1} + 2^q-t-a_{i+1}+a_{i+1}=d$ if $\sigma(i+1)=i+1$. On the other hand, $\sigma(i+1)\geq i+2$ forces $e(u_{\sigma(i+1)})\geq d-a_{i+1}+m_{i+2}+a_{i+2}\geq d+t+a_{i+1}>d$. Thus $P=u_1^d\cdots u_{i+1}^d \cdot P_{k,i+2}^{2^{q+i}}\cdot \left(\sum_{\tau\in \mathfrak{S}_{k-i-1}}u_{\tau(i+2)}^{2^{k-i-1}}\cdots u_{\tau(k)}^2\right)^{2^q-t-a_{i+1}}\cdot P_{k,i+2}^{m_{i+2}}\cdot u_{i+2}^{a_{i+2}}\cdot h_{i+2}$ and hence 
$P=u_1^d\cdots u_{i+1}^d \cdot P_{k,i+2}^{2^{q+i}}\cdot P_{k,i+2}^{2^{q+1}-2t-2a_{i+1}}\cdot P_{k,i+2}^{m_{i+2}}\cdot u_{i+2}^{a_{i+2}}\cdot h_{i+2}.$ As $m_{i+2}=2^{q+i} + t+ 2a_{i+1}-a_{i+2}$, the inductive step is complete. 
 \end{proof}
 
 As a special case of Theorem 6.2, we observe that letting $t=2^q$ in \eqref{eq:6.10} -- in which case $U(m;k)$ gives the weakest upper bound for the Gr\"unbaum--Ramos -- gives an immediate strengthening of the Ham Sandwich Theorem in dimensions a power of two. 

\begin{corollary} \label{cor:6.3} Let $q\geq 0$. Given any $2^{q+k-1}$ masses $\mu_1,\ldots, \mu_{2^{q+k-1}}$ on $\mathbb{R}^{2^{q+k-1}}$, there exists $k$ hyperplanes $H_1,\ldots, H_k$ such that
$H_i,\ldots, H_k$ equipartitions $\mu_1,\ldots, \mu_{2^{q+i-1}}$ for all $1\leq i \leq k$. \end{corollary}

\subsection{Inclusion of Orthogonality Constraints} 

Whenever $t\geq 2$ and $q\geq 1$, the above argument can be modified so as to include full orthogonality, so that when $a_1=0$ one has a strengthening of $\Delta^\perp(m;k)\leq U(m;k)$ from Theorem \ref{thm:5.2} for all $m+1$ which is not a power of 2. While we prove this only when $2^q\geq a_{k-1}+t+k-3$, it is clear that adjustments in $\mathbf{a}$ and $\mathbf{m}$ below allow for full orthogonality plus constraints when $2^q=a_{j-1}+t+j-3$ and $2\leq j \leq k-1$. 

\begin{theorem} \label{thm:6.4} Let $t\geq 2$. If (i) $a_1\leq \cdots \leq a_k$,  (ii) $0\leq a_2\leq 2a_1+t-1$, and (iii) $0\leq a_{k-1} \leq 2^q-t-k+3$, then \begin{equation} \label{eq:6.14} \Delta^\perp (\mathbf{m},\mathbf{a};k)=2^q \cdot [ 2^{k-1} + 1] -t \,\,\, \text{and} \,\,\, k\Delta^\perp(\mathbf{m},\mathbf{a};k) = C(\mathbf{m},\mathbf{a},\mathcal{O};k)), \end{equation}
\noindent where $\mathbf{m}=(m_1,\ldots, m_k)$ is given by 
\begin{equation} \label{eq:6.15} m_1=2^{q+1}-t-a_1 \,\,\, and \,\,\, m_i=2^q\cdot [2^{i-2}-1]+t+i-3+2a_{i-1}-a_i \,\,\, for\,\, all\,\, 2\leq i \leq k. \end{equation} 
\noindent Thus $\Delta^\perp(\mathbf{m},\mathbf{a};k)=U(m_1;k)$ if $a_1=0$, and in particular  \begin{equation} \label{eq:6.16} \Delta^\perp(2^{q+1}-t,t-1, 2^q+t, 3\cdot 2^q +t, \cdots, 2^q\cdot [2^{k-2}-1]+t+k-3);k)=U(2^{q+1}-t;k).  \end{equation}

\end{theorem} 
\begin{proof} We shall consider the cases (I) $a_{k-1}\leq 2^q-t-k+2$ and (II) $a_{k-1}=2^q-t-k+3$ separately. For either, we define \begin{equation} \label{eq:6.17} P_{k,i}(\mathcal{O}):=\sum_{\psi \in \mathfrak{S}_{k-i+1}} u_{\psi(i)}^{k-1} \cdots u_{\psi(k)}^{i-1} \end{equation} for each $1\leq i\leq k$.  We have \begin{equation} \label{eq:6.18}  P= P_{k,1}^{m_1}\cdot u_1^{a_1} \cdot P_{k,1}(\mathcal{O})\cdot P_{k,2}^{m_2}\cdot u_2^{a_2}\cdots P_k^{m_k}\cdot u_k^{a_k}. \end{equation} 

\noindent Again, let $d=2^q\cdot [2^{k-1}+1]-t$. In case (I), we show via induction as before that \begin{equation} \label{eq:6.19}  P= u_1^d\cdots u_i^d\cdot P_{k,i+1}^{2^{q+i}+2^q-t-i-a_{i+1}}\cdot u_{i+1}^{a_{i+1}} \cdot P_{k,i+1}(\mathcal{O}) \cdot h_{i+1} \end{equation}  for all $0\leq i \leq k-1$,  where $u_0$ and the $h_i$ are the same as in the proof of Theorem \ref{thm:6.2}. As there, letting $i=k-1$ in \eqref{eq:6.19} immediately yields $P=u_1^d\cdots u_k^d$, and \eqref{eq:6.19} holds trivially when $i=0$. Assuming it true for $0\leq i\leq k-2$, we again consider $P_{k,i+1}^{2^{q+i}}\cdot P_{k,i+1}^{2^q-t-i-a_{i+1}}$, noting that $2^q-t-i-a_{i+1}\geq 0$ by assumption (i) of Theorem \ref{thm:6.4}. Thus $e(u_{\sigma(i+1)})\geq 2^q\cdot 2^{k-1} + 2^q-t-a_{i+1}-i+a_{i+1}+i=d$ if $\sigma(i+1)=(i+1)$, and the assumption that $t\geq 2$ shows that $e(\sigma(i+1))\geq d -a_{i+1} + m_{i+2}+a_{i+2}>d$ otherwise. Thus $P=u_1^d\cdots u_{i+1}^d \cdot P_{k,i+2}^{2^{q+i}}\cdot P_{k,i+2}^{2^{q+1}-2t-2i-2a_{i+1}}\cdot P_{k,i+2}^{m_{i+2}}\cdot u_{i+2}^{a_{i+2}}\cdot P_{k,i+2}(\mathcal{O}) \cdot h_{i+2}$. 

For case (II), the argument from case (I) shows that \eqref{eq:6.19} holds for all $0\leq i \leq k-2$. Thus $P=u_1^d\cdots u_{k-2}^d\cdot P_{k,k-1}^{2^{q+2}-1}\cdot u_{k-1}^{a_{k-1}}\cdot P_{k,k-1}(\mathcal{O})\cdot h_{k-1}$. Now the exponent $e(u_{k-1})$ will be strictly less than $2^q\cdot [2^{q-1}+1]-t=2\cdot 2^{q+2} -2 +a_{k-1}+k-1$ unless each of the resulting $2^{q+2}$ permutations determined by $P_{k,k-1}^{2^{q+2}-1}$ and $P_{k,k-1}(\mathcal{O})$ is the identity, so again one has $P=u_1^d\cdots u_k^d$. \end{proof}

\noindent \textbf{Remark 4.} It should be noted that Proposition \ref{prop:6.1} does not yield $\Delta^\perp(m;k)\leq U(m;k)$ if $m=2^{q+1}-1$ and $q\geq 1$. Indeed, for  $Q:=[\sum_\sigma u_{\sigma(1)}^{2^{k-1}}\cdots u_{\sigma(k)}^1]^{2^q}\cdot [\sum_\sigma u_{\tau(1)}^{2^k-1}\cdots u_{\tau(k)}^{2^k-1}]^{2^q-1}\cdot [\sum_\psi u_{\psi(1)}^{k-1}\cdots u_{\psi(k))}^{0}]$ in $\mathbb{Z}_2[u_1,\ldots, u_k]/(u_1^{d+1},\ldots, u_k^{d+1})$, $d=U(m;k)$, it is clear that $e(u_{\sigma(1)})\geq 2^{q+k-1}+2^q-1+0=d$, so that $e(u_{\sigma(2)})\geq 2^{q+k-2}+2^{q+k-2}+2^q-2+1=d$ as well. Thus $Q$ consists of sums of the form $x_{\sigma(1)}^dx_{\sigma(2)}^d\cdot \Pi_{j\neq \sigma(1), \sigma(2)}u_j^{\alpha_j}$ for each $\sigma\in\mathfrak{S}_k$ and must therefore vanish. Therefore the estimate $8\leq \Delta^\perp (3;3)\leq 9=U(3;3)$ of Theorem \ref{thm:5.2} is not recovered from Proposition \ref{prop:6.1}, and one can check that neither $\Delta^\perp(1;2)=2=U(1;2)$ nor $\Delta^\perp(1;3)=4=U(1;3)$ is recovered either. Nonetheless, Proposition \ref{prop:7.4} below shows in particular that $6\leq \Delta^\perp(1;4)\leq 8=U(1;4)$, and one can easily extend this to $\Delta^\perp(1;k)$ for all remaining $k$. More interestingly, the following proposition (stated only for $\mathbf{a}=\mathbf{0}$ but true under the same assumptions as in Theorem \ref{thm:6.4}) shows that whenever $t\geq1$ one can still prescribe orthogonality on all but one pair of hyperplanes and still let $m_2=t$ as in Theorem \ref{thm:6.2}. 

\begin{proposition} \label{prop:6.5}  Let $(1,2)^C=\{(r,s)\mid 1\leq r<s\leq k\,\, \text{and}\,\, (r,s)\neq (1,2)\}$. If $1\leq t\leq 2^q$ and $2^q\geq t+k-3$, then \begin{equation} \label{eq:6.20} \Delta(\mathbf{m}, (1,2)^C;k)= U(m_1;k) \,\,\, \text{and} \,\,\, k\Delta(\mathbf{m},(1,2)^C;k) = C(\mathbf{m},\mathcal{O};k)), \end{equation} where
 \begin{equation} \label{eq:6.21} m_1=2^{q+1}-t, m_2=t, m_3=2^q + t -2,\,\, \text{and}\,\, m_i=2^q\cdot[2^{i-2}-1]+t+i-3\,\, \text{for all}\,\, i\geq 4.\end{equation}
\end{proposition} 

\begin{proof} Let $\mathcal{O}=(1,2)^C$, so that $P=P_{k,1}^{m_1}\cdot (u_1+u_3)\cdots (u_1+u_k)\cdot P_{k,2}^{m_2} \cdot P_k(\mathcal{O}_2)\cdot P_{k,3}^{m_3} \cdots P_{k,k}^{m_k}$. Manipulating as before shows that $P=u_1^d\cdot u_3\cdots u_k \cdot P_{k,2}^{2^{q+1}+2^q-t}\cdot P_k(\mathcal{O}_2)\cdot P_{k,3}^{m_3}\cdot P_{k,4}^{m_4}\cdots P_{k,k}$, and so that $P=u_1^du_2^d\cdot x_3\cdots x_k\cdot P_{k,3}^{2^{q+2}+ 2^q-t-2}\cdot [\sum_\sigma u_{\sigma(3)}^{k-2}\cdots u_{\sigma(k)}^1]\cdot P_{k,4}^{m_4}\cdots P_{k,k}=u_1^du_2^d\cdot P_{k,3}^{2^{q+2}+2^q-t-2}\cdot P_{k,3}(\mathcal{O})\cdot P^{m_4}_{k,4}\cdots P_{k,k}$. The proof is then identical to that of Theorem \ref{thm:6.4}.\end{proof} 

\noindent Before proceeding to explicit computations, we give one final variant of Theorem \ref{thm:6.2} (again stated only for $\mathbf{a}=\mathbf{0}$), which for all $m\geq 1$ and $k\geq 3$ allows any  $1\leq j \leq k-1$ of $H_1,\ldots, H_{k-1}$ to be orthogonal to $H_k$ in the same dimension, so long as $H_k$ bisects $j$ fewer masses.

\begin{proposition} \label{prop:6.6}  Let $1\leq t\leq 2^q$, let $k\geq 3$, let $\langle k \rangle^\perp: = \{(r,k)\mid 1\leq r<k\}$, and let $\mathcal{O}\subseteq \langle k \rangle^\perp$. If $|\mathcal{O}|=j$ and $1\leq j\leq k-1$, then   
  \begin{equation} \label{eq:6.22} \Delta(\mathbf{m},\mathcal{O};k)=U(m_1;k) \,\, \text{and}\,\, k\Delta(\mathbf{m},\mathcal{O};k)=C(\mathbf{m},\mathbf{a}, \mathcal{O};k), \end{equation}
\noindent where $\mathbf{m}=(2^q-t, t, 2^q+t, 3\cdot 2^q+t,\ldots 2^q\cdot[2^{k-2}-1]+t-j)$.
\end{proposition}

\begin{proof} It suffices to consider $\mathcal{O}=\langle k\rangle^\perp$, in which case $P=P_{k,1}^{m_1}\cdot  (u_1+u_k) \cdot P_{k,2}^{m_2}\cdot (u_2+u_k)\cdots P_{k,k-1}^{m_{k-1}}\cdot (u_{k-1}+u_k)\cdot P_{k,k}^{m_k}$. A straightforward induction shows that $P=u_1^d\cdots u_i^d \cdot P_{k,i+1}^{2^{q+i}+2^q-t}\cdot u_k^i \cdot (u_{i+1}+u_k)\cdots (u_{k-1}+u_k) \cdot h_{i+1}$ for all $i\leq k-2$, so that $P=u_1^d\cdots u_{k-2}^d \cdot P_{k,k-1}^{2^{q+k-2}+2^q-t}\cdot (u_{k-1}+u_k)\cdot u_k^{m_k+k-2}$ and again $P=u_1^d\cdots u_k^d$.
\end{proof}

\section{Examples} 

We conclude with some particular cases of Theorems \ref{thm:6.2} and \ref{thm:6.4} and Propositions \ref{prop:6.5}--\ref{prop:6.6} for $2\leq k \leq 4$, with comparison to the best known estimates for $\Delta(m;k)$. While each of the results from the previous section is tight in that equality holds in \eqref{eq:2.1} from Theorem \ref{thm:2.1}, with the original Gr\"unbaum--Ramos problem in mind we shall focus on those which place primacy on the ``fullness" of the equipartition. 

	First, we shall say that a constrained equipartition is \textit{optimal} if 
\begin{equation} \label{eq:7.1} L(m_1;k) \leq \Delta(\mathbf{m},\mathbf{a},\mathcal{O};k)<L(m_1+1;k), \end{equation}
\noindent where $L(m;k)$ denotes the Ramos lower bound. In the case of equality on the left hand-side of \eqref{eq:7.1}, the constrained equipartition strengthens $\Delta(m_1;k)$, as is the case with $\Delta^\perp(2^{q+1}-2; 2)$ of Theorem \ref{thm:5.2}. The inequality on the right hand-side of \eqref{eq:7.1} guarantees that the full equipartition of any $m_1+1$ measures in dimension $\Delta(\mathbf{m},\mathbf{a},\mathcal{O})$ is impossible, as is the case with $\Delta^\perp(1;3)$, $\Delta^\perp(2^{q+2}-2;2)$, and $\Delta^\perp(2^{q+2};2)$ obtained in Theorem \ref{thm:5.2}. 

	To evaluate the fullness of the cascade condition $\mathbf{m}$ of the constrained equipartition, for each $1\leq i \leq k$ we let  $\mathbf{m}_i=(m_1,\ldots, m_i,0,\ldots, 0)$ and we let $L(\mathbf{m}_i+\mathbf{e}_i,\mathbf{a},\mathcal{O};k):=\left \lceil \frac{C(\mathbf{m}_i+\mathbf{e}_i,\mathbf{a},\mathcal{O};k)}{k}\right\rceil$ denote the corresponding lower bound from Theorem 2.1.  We shall say that that $\Delta(\mathbf{m},\mathbf{a},\mathcal{O};k)$ is \textit{maximal at the $i$-th stage}  if \begin{equation} \label{eq:7.2}
 \Delta(\mathbf{m}, \mathbf{a}, \mathcal{O}; k)<L(\mathbf{m}_i+\mathbf{e}_i,\mathbf{a},\mathcal{O};k), \end{equation}  \textit{$j$-maximal} if \eqref{eq:7.2} holds for all $1\leq i \leq j$, and \textit{maximal} if $j=k$.  In particular, a pure cascade is 1-maximal iff it is optimal. Observe also that maximality of a cascade at the $i$-th stage requires $m_{i+1}\leq 2$ (though not necessarily conversely), and we shall say that $\Delta(\mathbf{m},\mathbf{a}, \mathcal{O};k)$ is \textit{balanced} if this weaker condition is satisfied for all $1\leq i \leq k-1$. 
 	
\subsection{$k=2$} Letting $t=1,2, \, \text{and}\, 3$ in Theorems \ref{thm:6.2} and \ref{thm:6.4} immediately yield the following:

\begin{corollary} \label{cor:7.1} Let $q\geq0$.\\
$\bullet$ $\Delta((2^{q+1}-1,1);2)=3\cdot 2^q - 1$\\
$\bullet$ $\Delta((2^{q+2}-2, 2);2)=\Delta^\perp((2^{q+2}-2),1); 2))=3\cdot 2^{q+1}-2$\\
$\bullet$ $\Delta^\perp((2^{q+3}-3, 2);2)=3 \cdot 2^{q+2}-3$\end{corollary}

Note that the first equation gives an alternative strengthening to $\Delta(2^{q+1}-1;2)$ than the orthogonality of Theorem \ref{thm:5.2}, while the orthogonal case from the second line strengthens the exact value of $\Delta^\perp(2^{q+2}-2;2)$ obtained there.  While each of those from the first two lines are optimal and maximal, observe that $\Delta^\perp((2^{q+3}-3,2);2)=L(2^{q+3}-2;2)$ is maximal but not optimal. 

\subsection{$k=3$} After decreasing $m_3$ whenever $m_3>2$,  we have the following balanced cases when $q=0, 1,$ or $2$ in Theorems 6.2 and 6.4 and Propositions 6.5--6.6:

\begin{corollary} \label{cor:7.2} $\newline$
$\bullet$ $\Delta((1,1,2);3)=\Delta((1,1,1),\{2,3\};3)=\Delta((1,1,0),\langle 3\rangle^\perp;3)=4$\\
$\bullet$  $\Delta((3,1,2),\{2,3\};3)=\Delta((3,1,1), \langle 3 \rangle^\perp;3)=9$\\
$\bullet$  $\Delta^\perp((2,1,2);3)=\Delta((2,2,2), \langle 3\rangle^\perp;3)=8$\\
$\bullet$ $18\leq \Delta((7,1,2), \{2,3\};3),\, \Delta((7,1,1),\langle 3\rangle^\perp;3) \leq 19$\\
 $\bullet$ $17 \leq \Delta^\perp((6,1,2);3)\leq  18$ and $\Delta((6,2,2), \langle 3\rangle^\perp;3)=18$\\
\end{corollary}

All of the terms from the first and second lines of equalities in Corollary \ref{cor:7.2} are both optimal and maximal and should be compared to the known values $\Delta(1;3)=3$, $\Delta(2;3)=5$, and $\Delta(4;3)=10$. Note also that $\Delta((1,1,0);3)=4$ and $\Delta((3,1,1);3)=9$ are already optimal. For the remaining three equalities, $\Delta^\perp((2,1,2);3)$ and $\Delta((6,2,2), \langle 3\rangle^\perp;3)$ are only balanced, while $\Delta((2,2,2), \langle 3\rangle^\perp;3)$ is in addition maximal at the second stage. 

 \subsection{$k=4$} Again, balanced cases arise by letting $q=0$ and $q=1$ in Theorem \ref{thm:6.2} and Propositions \ref{prop:6.5}--\ref{prop:6.6} (and decreasing $m_4$ if $m_4>2$). 
 
\begin{corollary} \label{cor:7.3} $\newline$
$\bullet$ $\Delta((1,1,2,2),\{(2,4),(3,4)\};4)=\Delta((1,1,2,1),\langle 4\rangle^\perp ;4)=8$\\
$\bullet$ $16\leq \Delta((3,1,1,2), (1,2)^C;4) \leq 17$ 
\end{corollary} 

 In particular, $\Delta((1,1,2,1),\langle 4\rangle^\perp; 4)=8$ is 1-maximal, and one can note that  $\Delta((1,1,2,2);4)=8$ already holds without imposing any orthogonality.\\

We close with an estimate on full orthogonality in the case of a single measure and four hyperplanes, a result which cannot be derived from Lemma \ref{lem:5.1} using any currently known estimates of $\Delta(m;k)$. 
 
 \begin{proposition} \label{prop:7.4} For any four masses $\mu_1,\mu_2, \mu_3, \mu_4$ on $\mathbb{R}^8$, there exist four pairwise orthogonal hyperplanes $H_1,H_2, H_3, H_4$ such that\\
 $\bullet$ $H_1, H_2, H_3, H_4$ equipartition $\mu_1$,\\
 $\bullet$ any two of $H_2,H_3,H_4$ equipartition $\mu_2$, and\\
 $\bullet$ each of $H_3$ and $H_4$ simultaneously bisect $\mu_3$ and $\mu_4$,\\
$\bullet$ and $H_4$ contains any prescribed point. 

\noindent In particular,
  \begin{equation} \label{eq: 7.3} 6 \leq \Delta^\perp(1;4)\leq 8. \end{equation} \end{proposition}

\begin{proof} The representation arising from the Makeev-type condition [15, Theorem 4] that any two of $H_2,H_3,H_4$ equipartition $\mu_2$ is $\oplus_{i=2}^4V_{\mathbf{e}_i}\oplus_{2\leq r<s\leq 4} V_{\mathbf{e}_r+\mathbf{e}_s}$, whose corresponding polynomial is $P_{4,2}(\mathcal{O})$ \eqref{eq:6.17}. Thus 
		\begin{equation} \label{eq:7.4} P=P_{4,1} \cdot P_4(\mathcal{O}) \cdot P_{4,2} (\mathcal{O}) \cdot u_3^2u_4^3. \end{equation}
One therefore has $P=u_1^8\cdot P_{4,2}\cdot P_{4,2}^2(\mathcal{O})\cdot u_3^2u_4^3=u_1^8[\sum_\sigma u_{\sigma(2)}^4u_{\sigma(3)}^2u_{\sigma(4)}^1]\cdot [\sum_\tau u_{\tau(2)}^6u_{\tau(3)}^4u_{\tau(4)}^2]\cdot u_3^2u_4^3$, so that $\tau(2)=\sigma(3)=2$ and $P=u_1^8u_2^8\cdot [\sum_\sigma u_{\sigma(3)}^4u_{\sigma(4)}^1]\cdot [\sum_\tau u_{\tau(3)}^4u_{\tau(4)}^2]u_3^2u_4^3$. Hence $\sigma(3)=\tau(4)$ and $\sigma(4)=\tau(3)$, so $P=u_1^8u_2^8\cdot [\sum_\sigma u_{\sigma(3)}^6u_{\sigma(4)}^5]\cdot u_3^2u_4^3=u_1^8u_2^8u_3^8u_4^8$.
\end{proof}

As $L(2;4)=8$, note that the estimate $6\leq \Delta^\perp(1;4)\leq 8$ would follow from $\Delta(2;4)$ if the Ramos conjecture holds, as $\Delta^\perp(m;k)\leq \Delta(m+1;k)$ follows in general by letting one of the masses be a ball with uniform density (see, e.g., [18]). At present, however, the best estimate is $8\leq \Delta(2;4)\leq 10$ from [5]. Moreover, the best upper bound obtainable from Lemma \ref{lem:5.1} is $\Delta^\perp(1;4)\leq 9$, which occurs by using $m=4,\, k=3,$ and $d=10$ in part (b): Bisecting a given mass in $\mathbb{R}^9$ by a hyperplane $H$, one can use $U_3\oplus U_3 \oplus W_3$ to find $H_1,H_2,$ and $H_3$ which equipartition the resulting two masses. On the other hand, $\oplus_{i=1}^3V_{\mathbf{e}_i}\oplus_{1\leq r<s\leq 3} V_{\mathbf{e}_r+\mathbf{e}_s} $ in the ``remainder" $U_3\oplus V_{\alpha\in \mathcal{B}^C}$ can be used to ensure that $H\perp H_i$ for all $1\leq i \leq 3$ and that $H_1,H_2,$ and $H_3$ are pairwise orthogonal, respectively.

\section*{Acknowledgements}

The author thanks the anonymous reviewers for their thoughtful suggestions and comments which improved the exposition of this paper, as well as Florian Frick for very helpful conversations.

\bibliographystyle{plain}

\end{document}